\let\iftxfonts=\iftrue  
\let\iflabels=\iffalse  
\let\iflitnum=\iftrue   
\let\ifsrcltx=\iffalse  

\iflabels
\documentclass[11pt,reqno]{amsart}
\else
\documentclass[11pt]{amsart}
\fi

\usepackage{amssymb,amsmath}
\usepackage[all]{xy}
\usepackage{graphicx}

\iflabels
  
  \usepackage{showlabels}
\fi

\iftxfonts
\usepackage[varg]{txfonts}
\fi

\ifsrcltx
\usepackage[active]{srcltx}
\fi

\def\be#1\ee{\begin{equation}#1\end{equation}}
\def\bad#1\ead{\be\begin{aligned}#1\end{aligned}\ee}
\def\badl#1#2\eadl{\be\label{#1}\begin{aligned}#2\end{aligned}\ee}
\renewcommand\={\,=\,} \numberwithin{equation}{section}

\newtheorem{thm}{Theorem}[section]
\newtheorem{prop}[thm]{Proposition}

\newtheorem{lem}[thm]{Lemma} \theoremstyle{definition}
\newtheorem{defn}[thm]{Definition}

\newcommand\Gm{\Gamma}
\newcommand\gm{\gamma}

\newcommand\dt{\delta}

\newcommand\ph{\varphi}

\newcommand\Ps{\Psi}
\newcommand\ps{\psi}
\newcommand\Om{\Omega}

\newcommand\CC{\mathbb{C}}

\newcommand\RR{\mathbb{R}}
\newcommand\ZZ{\mathbb{Z}}

\newcommand\SL{{\mathrm{SL}}}
\newcommand\uhp{{\mathfrak H}}
\newcommand\lhp{{\mathfrak H^-}}

\newcommand\proj[1]{\mathbb {P}^1_{#1}}

\newcommand\im{\mathrm{Im}\,}
\renewcommand\setminus{\smallsetminus}
\newcommand\oh{{\mathrm O}}

\newcommand\cf[1]{S_{\!#1}}
\newcommand\bcf[1]{\bar S_{\!#1}}
\newcommand\dsv[2]{\mathcal{D}_{#1}^{#2}}

\newcommand\Coh{{\mathbf{Coh}}}

\newcommand\ff{{\textbf f}}
\renewcommand\O{{\mathbf O}}
\newcommand\N{{\mathbf N}}
\newcommand\A{{\mathcal A}}
\newcommand\B{{\mathcal B}}
\newcommand\X{{\mathsf X}}
\newcommand\Y{{\mathsf Y}}

\renewcommand\v{{\mathbf v}}
\newcommand\w{{\mathbf w}}
\newcommand\h{{\mathbf h}}

\makeatletter
\newcommand\matc[4]{\left( {#1\@@atop #3}{#2\@@atop #4}\right)}
\newcommand\matr[4]{\left( {\hfill #1\@@atop\hfill #3}{\hfill
#2\@@atop\hfill #4}\right)}
\makeatother

\newcommand\rmrk[1]{\medskip\par\noindent\emph{#1.
}}
\newcommand\rmrkn[1]{\medskip\par\noindent\emph{#1
}}

\newcommand\rmrks{\medskip\par\noindent\emph{Remarks.
}}
\newcounter{rmrkcnt}
\renewcommand\thermrkcnt{(\alph{rmrkcnt})}
\newcommand\itmi{\setcounter{rmrkcnt}{0}%
\refstepcounter{rmrkcnt}\thermrkcnt\ \ }
\newcommand\itm{\smallskip\par\noindent%
\refstepcounter{rmrkcnt}\thermrkcnt\ \ }

\usepackage{color}

\definecolor{blue}{rgb}{0,0,1}
\definecolor{red}{rgb}{1,0,0}

\long\def\red#1\endred{\textcolor{red}{#1}}
\long\def\blue#1\endblue{\textcolor{blue}{#1}}

\begin{document}
\title[Multiple period integrals and cohomology]{Multiple period integrals
and cohomology}
\author{R.\,Bruggeman}
\address{Mathematisch Instituut Universiteit Utrecht, Postbus 80010, 3508 TA
Utrecht, Nederland}
\email{r.w.bruggeman@uu.nl}
\author{Y.\,Choie}
\address{Dept.~of Mathematics and PMI, Postech, Pohang, Korea  790--784}
\email{yjc@postech.ac.kr }
\date{}
\begin{abstract}
This work  gives   a version of the Eichler-Shimura isomorphism
   with    a non-abelian $H^1$  in group cohomology. Manin
has given a map from vectors of cusp forms to a noncommutative cohomology
set by means of iterated integrals.   We show 
 Manin's map is injective but far from
surjective.  By extending Manin's map we are able to construct 
a bijective map and remarkably this   establishes  the existence of a
 non-abelian
 version of the Eichler-Shimura map. 
\end{abstract}



\thanks{\emph{Acknowledgement. }The second author is partially supported by
NRF-2015049582 and NRF-2013R1A2A2A01068676. }

\maketitle

\section{Introduction} In the theory of modular forms the Eichler-Shimura
isomorphism has played an important role, with many applications. For
instance, it gives integrality of eigenvalues for Hecke operators and
algebraicity of the critical values of the $L$-functions of modular forms.
 That enables the construction of$p$-adic $L$-functions, gives a connection
to Iwasawa theory as well as the computational aspects of modular form
theory, etc. The Eichler-Shimura isomorphism relates spaces of cusp forms
of integral weight to a parabolic cohomology group, namely,
$$\cf{k} (\SL_2(\mathbb{Z}))  \oplus \bcf k(\SL_2(\mathbb{Z}))
 \cong H^1_{\tiny{par}} (\Gamma, \mathbb{C}_{k-2}[X,Y]),$$
where $\mathbb{C}_{k-2}[X,Y]$ is the $\SL_2(\mathbb{Z})$-module of
homogeneous polynomials of degree $k-2$ in the indeterminates $X, Y$, and
$S_k(\SL_2(\mathbb{Z})) $
(resp. $\overline{S }_k(SL_2(\mathbb{Z}))$) is the space of holomorphic
(resp. antiholomorphic) cusp forms of weight $k.$

  Knopp and Mawi \cite{Kn74,KM10} eventually extended Eichler-Shimura
  isomorphism by establishing a canonical isomorphism between $1$-cohomology
 of cofinite discrete subgroups $\Gamma$ of $\SL_2(\mathbb{R})$ with
 appropriate holomorphic coefficients and the space of cusp forms with real
 weight.

In $2005$ Manin \cite{Ma5} defined a "non-abelian"   $H^1$
in group cohomology with values in a non-abelian group, and a 
map from a product of spaces of cusp forms to
 this cohomology set, in analogy to the Eichler-Shimura map. Manin's
construction uses iterated integrals in the spirit of the multiple zeta
values, which have proved so useful to understand zeta values, mixed Tate
 motives etc. Manin's integrals gives a way to express multiple $L$-values
 of modular forms and have been studied by the second author \cite{Ch14} and
independently in the thesis of Provost \cite{Pro} recently. In \cite{Ch14} 
the period polynomial whose coefficients are the 
multiple $L$-values were treated as an elements in non-ableian $H^1$ first time.  


In this recent 2014-ICM talk, Brown \cite{Bro} mentioned a connection
between   the   iterated integrals of Manin  and 
certain mixed motives. He explained
how to interpret motivic multiple zeta values as periods of the
pro-unipotent fundamental groupoid of the projective line minus three
 points $X=\mathbb{P}^1\backslash \{0,1,\infty\}$ via iterated integral of
smooth $1$-forms on a differentiable manifold discussed by Chen \cite{Cng}.
Also the recent work by Hain \cite{Hai} discusses the relation between
Manin's iterated integrals and the Hodge theory of modular groups. However,
  it was not clear yet how to relate Manin's iterated integral and
  Eichler-Shimura theory. Manin's map from spaces of cusp forms to cohomology
 differs in two aspects from the Eichler-Shimura map: the summand
 $ \bcf k(\SL_2(\ZZ))$ is absent and the map is injective but not
   surjective.
\medskip

This paper addresses the second  difference  by extending  Manin's map to more
complicated combinations of spaces of cusp forms to obtain a variant of the
 Eichler-Shimura isomorphism with values in a non-abelian cohomology $H^1.$
  Our main result (Theorem~\ref{thm-surj}) states that there is an extension
of Manin's map that is bijective onto a non-commutative cohomology set. 
It is remarkable that
  there exists some 
 non-abelian version of the Eichler-Shimura map.

To obtain our main result we modify Manin's construction \cite{Ma5} in
 several ways in the spirit of a variant Eichler-Shimura isomorphism
established in \cite{Kn74,KM10}: first replace the finite dimensional
spaces of polynomials by spaces of functions on the lower half-plane.
  Secondly, unlike in the classical Eichler-Shimura isomorphism,
antiholomorphic modular forms are not considered. Thirdly, we allow
automorphic forms for cofinite discrete subgroups $\Gamma$ of
$SL_2(\mathbb{R}) $ with arbitrary real weights and multiplier system.
Finally, we collapse the number of variables in the iterated integrals.
\medskip

To be more precise consider the iterated integral
\badl{Rdef0} 
 R_\ell
(f_1,&\ldots,f_\ell;y,x; t_1,\ldots,t_\ell)
\;:=\; \int_{ \tau_1= x }^{y} f_1(\tau_1)\,
  (\tau_1-t_1)^{w_1}\,
   \\
  &\hbox{} \cdot\int_{ \tau_2= x  }^{\tau_1}\, f_2(\tau_2) (\tau_2-t_2 )^{w_2}
\cdots
\int_{ \tau_\ell= x}^{\tau_{\ell-1}} f_\ell(\tau_\ell)\,
(\tau_\ell-t_\ell)^{w_\ell}\, d\tau_\ell \cdots
d\tau_2\, d\tau_1, \, \eadl
where $x, y $ are in the extended complex upper half-plane and each
 $t_j, 1\leq j \leq \ell,$ is in the lower half plane. If the $f_j$ are cusp
 forms of even integral weight $w_j+2$ the iterated integral defines a
polynomial function in the $t_j$, $1\leq j \leq \ell$ whose coefficients
are multiple $L$-values of $f_j.$ The resulting iterated integral is
holomorphic in $(t_1,\ldots,t_\ell)$ in the product of $\ell$ copies of the
lower half-plane if the $f_j$ are cusp forms of real weight.
As the order $\ell$ of the iterated integral increases the relations between
  iterated integrals become more and more complicated. However, the relations
between iterated integrals of order $\ell$ look simple modulo all products
of iterated integrals of lower order. Manin \cite{Ma5, Ma6} has shown how
to give a neat formulation for all relations  among  iterated integrals of
the type indicated in~\eqref{Rdef0}. His approach works with formal series
in non-commuting variables, and can be applied to much more general
iterated integrals than studied here.

The factors $(\tau_j-\nobreak t_j)^{w_j}$ in~\eqref{Rdef0} occur also in the
definition of cocycles attached to cusp forms. Manin attaches to vectors of
cusp forms $(f_1,\ldots,f_\ell)$ a cocycle in a noncommutative cohomology
set, and thus gives a generalization of the Eichler-Shimura map. The
cohomology has values in a non-commutative subgroup $\N(\A)$ of the unit
group of the non-commutative ring $\A$ of formal power series in
non-commuting variables $A_1,\ldots,A_\ell$ with coefficients in spaces of
holomorphic functions on the lower half-plane. The variables $A_j$
correspond to spaces of cusp form $\cf{w_j+2}\bigl( \Gm,v_j\bigr)$ with
positive real weights $w_j+2$ and corresponding multiplier systems $v_j$.
Then Manin's approach leads to a map
\be \prod_{j=1}^\ell \cf{w_j+2}\bigl( \Gm,v_j\bigr)
\longrightarrow H^1\bigl( \Gm;\N(\A)\bigr) \ee
from a product of finitely many spaces of cusp forms to a non-commutative
cohomology set. This (non-linear) map is far from surjective. In
  Theorem~\ref{thm-surj} we show that Manin's map can be extended, and that
all elements of the cohomology set $H^1\bigl( \Gm;\N(\A)\bigr)$ can be
related to combinations of cusp forms by means of iterated integrals. The
simplification $t_1=\cdots=t_\ell$ in the iterated integrals is essential
 for our methods to work.
\medskip

Sections \ref{sect-Kc} and~\ref{sect-ii} have a preliminary nature. We
review the approach of Knopp \cite{Kn74} to associate cocycles to any cusp
form of real weight and the definition of the iterated integrals that we
use. Sections \ref{sect-fs} and~\ref{sect-ncc} discuss Manin's approach to
use formal series in non-commuting variables to associate noncommutative
cocycles to vectors of cusp forms. In Section~\ref{sect-scf} we extend this
approach in such  a way that  the resulting map from 
collections of cusp forms to noncommutative cohomology is
bijective.\medskip

\noindent {\bf{Acknowledgement}} The authors would like to thank the
referees for numerours helpful comments and suggestions which greatly
improved the exposition of this paper, in particular the
introduction.

\section{Cusp forms and Theorem of Knopp and Mawi}\label{sect-Kc}

\rmrk{Discrete group} Let $\Gm$ be a cofinite discrete subgroup of
$\SL_2(\RR)$ with translations.  Without loss of generality we assume
that $\matr{-1}00{-1}\in \Gm$.  For convenience we conjugate $\Gm$
 into a position for which $\infty$ is among its cusps and such that
 the subgroup $\Gm_\infty$ of $\Gamma$ fixing $\infty$ is  generated
by  $T=\matc1101$.

\rmrk{Notation} For $w\in \RR$ and $v$ a corresponding \emph{unitary}
multiplier system we denote by $\cf{w+2}(\Gm,v)$ the space of holomorphic
cusp forms of weight $w+2$ and multiplier system $v$. This is the
finite-dimensional space of holomorphic functions $f$ on the upper
half-plane satisfying $f(\gm z) = v(\gm) \, (cz+\nobreak d)^{w+2}\, f(z)$
for $\gm \in \Gm$, with exponential decay upon approach of the cusps. If
the weight $w+2$ is integral a multiplier system is a character.

\rmrk{Functions with at most polynomial growth} By $V(v,w)$, with $w\in \RR$
and $v$ a corresponding multiplier system, we denote the space of
holomorphic functions on the lower half-plane $\lhp$ with at most
polynomial growth at the boundary $\proj\RR$ of $\lhp$, provided with the
action of $\gm=\matc abcd\in \Gm$ given by
\be f |_{v,-w}\gm\,(t) \= v(\gm)\,
 (ct+d)^w\, f(\gm t)   \,.\ee

The condition that $f$ has polynomial growth on $\lhp$ can be formulated as
\be\label{pg} f(t) = \oh\bigl(|t|^A\bigr)+ \oh\bigl( |\im
t|^{-A}\bigr)\qquad \text{ for all} \, t\in \lhp, \text{ for some }A\geq
0\,.\ee
 The action
$|_{v,-w}$ of $\Gamma$ preserves this condition.

\rmrks
\itmi
In \cite[\S1.4]{BCD} we denote the representation $V(v,w)$  of $\Gm$
by~$\dsv{v,-w}{-\infty}$ (actually we used $r=w+2$ as the main parameter,
and wrote $\dsv{v,2-r}{-\infty}$.)
\itm The polynomial growth condition in (\ref{pg}) can be formulated
 in terms of an estimate by one function
$Q(t) = |\im t|/|t-\nobreak i|^2$ as $f(t) = \oh\bigl( Q(t)^{-A}\bigr)$ for
some $A\geq 0$. See the discussion in \cite[\S1.5]{BCD}.
 
\rmrk{Knopp's cocycles associated to cusp forms} Knopp \cite{Kn74}
associated to cusp forms $f\in \cf {w+2}(\Gm,v)$ a cocycle $\bar\ps_f$
 given by
\[ \bar\ps_{f,\gm} (z) \=\overline{ \int_{\gm^{-1}\infty}^\infty f(\tau)\,
(\tau-\bar z)^w\, d\tau }\,.\]
This cocycle takes values in the holomorphic functions on the upper
half-plane $\uhp$ that have at most polynomial growth on $\uhp$
   in the sense of (\ref{pg})
  (now with $t$ replaced   by
$  z \in \uhp$). We avoid the complex conjugation by taking a cocycle with
values in the holomorphic functions on the lower half-plane $\lhp$ with at
most polynomial growth at the boundary:
\be\label{psidef} \ps_{f,\gm} (t) \= \int_{\tau=\gm^{-1}\infty}^\infty
f(\tau)\,
(\tau-t)^w\, d\tau\,.\ee
So $\ps_f$ has values in the $\Gm$-module $V(v,w)$.

\begin{thm}\label{thm-KM}{ \rm (Knopp and Mawi \cite{KM10}) }For real weight
$w+2$ and corresponding unitary multiplier system $v$ the map
$f \mapsto [\ps_f]$ determines a linear bijection
\[ \cf{w+2}(\Gm,v) \longrightarrow H^1\bigl(\Gm;V(v,w)\bigr)\,.\]
\end{thm}
Knopp \cite{Kn74} conjectured this result, and proved it for many cases.
Finally, the remaining cases were completed in~\cite{KM10}.

\rmrks
\itmi A multiplier system  $v$ is called \emph{unitary}
 if  $|v(\gm)|=1, \forall \gm\in \Gamma.$

\itm Since $\cf{w+2}(\Gm,v)=\{0\}$ for $w+2\leq 0$, the theorem implies that
the cohomology groups vanish as well for $w+2\leq 0$.

\itm If $w\in \ZZ_{\geq 0}$ the cocycles take values in polynomial functions
on $\lhp$  , which for the trivial multiplier system form a submodule of
$V(1,w)$ isomorphic to $\CC_w[X,Y]$.  

If the multiplier system $v$ has values only in $\{1,-1\}$ then conjugation
gives cocycles in the same module. The Eichler-Shimura theory gives the
parabolic cohomology group with values in polynomial functions of degree at
most $w$ as the direct sum of the images of the two maps
$f\mapsto [\ps_f] $ and $f\mapsto [\bar\ps_f]$. However, in the large
module of polynomially growing functions the cocycles $\bar\ps_f$ become
coboundaries. Also the cocycles associated to Eisenstein series become
coboundaries over the module of functions with at most polynomial growth.

\itm Knopp \cite{Kn74} shows also that the parabolic cohomology group
$H^1_{\tiny{par}}\bigl(\Gm;V(v,w)\bigr)$ is    equal   to   the cohomology group
$H^1\bigl(\Gm;V(v,w)\bigr)$.

\section{Iterated integrals}\label{sect-ii}
By taking
 $t_1=\cdots =t_\ell=t$   we  consider holomorphic function in $t$
 running through the lower half-plane:
\badl{Rdef1} R_\ell
(f_1,&\ldots,f_\ell;y,x; t)
\;:=\; \int_{\tau_1=x}^y f_1(\tau_1)\,
  (\tau_1-t)^{w_1}\,
   \\
  &\hbox{} \cdot \int_{\tau_2=x}^{\tau_1}\, f_2(\tau_2) (\tau_2-t )^{w_2}
\cdots
\int_{\tau_\ell=x}^{\tau_{\ell-1}} f_\ell(\tau_\ell)\,
(\tau_\ell-t)^{w_\ell}\, d\tau_\ell \cdots
d\tau_2\, d\tau_1\,. \eadl
It is a multilinear form on $\prod_{j=1}^\ell\cf{w_j+2}(\Gm,v_j)$ for
$\ell$~pairs $(v_1,w_1),\ldots,\allowbreak (v_\ell,w_\ell)$ of real numbers
$w_j$ and corresponding unitary multiplier systems~$v_j$. The parameter $t$
is in the lower half-plane~$\lhp$. The value of the iterated integral does
not depend on the path of integration, provided we take care to approach
 cusps along geodesic half-lines
(for instance vertically).

The most interesting case is 
$y=\gm^{-1}\infty$, $\gm\in \Gm$, and $x=\infty$. For $\ell=1$ this gives
the value $\psi_{f_1,\gm}$ of the cocycle in~\eqref{psidef}. That is why we
call $R_\ell(f_1,\ldots,f_\ell;\gm^{-1}\infty,\infty;t)$ a \emph{multiple
period integral}.

\rmrk{Functions with at most polynomial growth} The condition of polynomial
growth in \eqref{pg} is preserved by the action of $\Gm$ given for
$\gm=\matc abcd$ by
\badl{pact} h|_{\v,-\w}\gm (t) &\= \v(\gm)^{-1}\, (ct+d)^\w h(\gm t)\,,\\
\v(\gm) &\= v_1(\gm)v_2(\gm)\cdots v_\ell(\gm)\,,\qquad \w \=
w_1+w_2+\cdots+w_\ell\,.
\eadl

By $V(\v,\w)$ we denote the vector space of holomorphic functions on $\lhp$
  with the action $|_{\v,-\w}$ in~\eqref{pact}.   Multiplication of
functions gives a bilinear map
$V(\v;\w)\times V(\v';\w')\rightarrow V(\v\v';\w+\w')$. The action behaves
according to the following rule:
\be \label{mlta}
(h|_{\v,-\w}\gm)\;(h'|_{\v',-\w'}\gm) =
(hh')|_{\v\v',-\w-\w'}\gm\,.\ee

\begin{lem}\label{lem-pg}For
$\ff=(f_1,\ldots,f_\ell) \in \prod_{j=1}^\ell \cf{w_j+2}(\Gm,v_j)$ the
multiple period integral $R_\ell(\ff;y,x;\cdot) $
defines an element of $V(\v,\w)$.
\end{lem}
\begin{proof}Each cusp form has at most polynomial growth on $\uhp$, and has
exponential decay at cusps when the cusp is approached along a geodesic
half-line. This implies that the  iterated integral in  \eqref{Rdef1} has at
most polynomial growth in $t$ and $\tau_{\ell-1}$. Successively this also
implies polynomial growth in $\tau_{j-1}$ and $t$ of the further integrals.
\end{proof}

\rmrk{Trivial relation}Directly from the definition we have
\be\label{trivrel} R_\ell(\ff;x,x;t)=0\,.\ee

\begin{lem}\label{lem-inv}For $\gm\in \Gm$
\be R_\ell(\ff;\gm^{-1}y,\gm^{-1}x;t) \=
R_\ell(\ff;y,x;\cdot)|_{\v,-\w}\gm\,(t)\,.
\ee
\end{lem}
\begin{proof}In the following computation all $\tau_j$ are replaced by
$\gm\tau_j$, with $\gm=\matc abcd \in \Gm$.
 \begin{align*}
 R_\ell(\ff&;x,y;\cdot)|_{\v,-\w}\gm \,(t)
 \= \prod_{j=1}^\ell \Bigl( v_j(\gm)^{-1}\,
 (ct+d)^{w_j} \Bigr)
 \int_{\tau_1=x}^y f_1(\tau_1)\, (\tau_1-\gm t)^{w_1} \\
 &\quad\hbox{} \cdots
 \int_{\tau_\ell=x}^{\tau_{\ell-1}} f_\ell(\tau_\ell)\, (\tau_\ell  -\gm
 t )^{w_\ell}\, d\tau_\ell\cdots d\tau_1
 \displaybreak[0]
 \\
 &\= \prod_{j=1}^\ell \Bigl( v_j(\gm)^{-1}\,
 (ct+d)^{w_j}  \Bigr)
 \int_{\tau_1=\gm^{-1}x}^{\gm^{-1}y} f_1(\gm \tau_1)\,
 \frac{(\tau_1-t)^{w_1}}{
 (c\tau_1+d)^{w_1}\, (ct+d)^{w_1} }\\
 &\quad\hbox{} \cdots
 \int_{\tau_\ell=\gm^{-1}x}^{\gm^{-1}(\gm \tau_{\ell-1})} f_\ell(\gm
 \tau_\ell)\, \frac{(\tau_\ell-t)^{w_\ell}}{
 (c\tau_\ell+d)^{w_\ell}\,
 (ct+d)^{w_\ell} } \,\frac{d\tau_\ell}{(c\tau_\ell+d)^2}\cdots
 \frac{d\tau_1}{(c\tau_1+d)^2}\displaybreak[0]\\
 &\=R_\ell(\ff;\gm^{-1}y,\gm^{-1}x;t)\,.
 \end{align*}
 \end{proof}

\rmrk{Cocycles}For $\ell=1$ we get the cocycle $\ps_f$ in \eqref{psidef}:
\be\label{o1coc}
\ps_{f,\gm}(t) = - R_1(f;\gm^{-1}\infty,\infty;t)\,. \ee

\rmrk{Decomposition} It is easy to see that the cocycles in \eqref{psidef}
satisfy the cocycle relation $c_{\gm\dt}=c_\gm|\dt+c_\dt$ for
$\gm,\dt\in \Gm$: Use the decomposition relation
$\int_b^a + \int_c^b = \int_ c^a$ for integrals together with the
invariance relation in~Lemma~\ref{lem-inv}.

  There are \emph{decomposition relations} for the iterated integrals in
\eqref{Rdef1}, which can be obtained by application of the decomposition
relation for integrals of one variable to the subintegrals
in~\eqref{Rdef1}. For the orders $2$ and $3$ these relations take the
following form:
 \begin{align}\label{rel2}
R_2(f_1,&f_2;z,y;t) + R_2(f_1,f_2;y,x;t) \\
\nonumber
&\quad\hbox{}- R_2(f_1,f_2;z,x;t)
\= R_1(f_1;z,y;t)\; R_1(f_2;y,x;t)\,,
\displaybreak[0]\\
\label{rel3}
R_3(f_1,&f_2,f_3;z,y;t) + R_3(f_1,f_2,f_3;y,x;t)\\
\nonumber
&\qquad\quad\hbox{}
- R_3(f_1,f_2,f_2;z,x;t)\\
\nonumber
& \=
-R_1(f_1;z,y;t)\; R_2(f_2,f_3;y,x;t)\\
\nonumber
&\qquad\hbox{}
+ R_2(f_1,f_2;z,y;t)\; R_1(f_3;y,x;t)\,.
  \end{align}
We have written these relations in such a way that the quantity on the left
should be zero if the standard decomposition would hold. On the right is a
correction term consisting of products of iterated integrals of lower
order.

\rmrk{Example}The decomposition relations can be used to obtain relations
between values of multiple $L$-functions at special points, as studied
in~\cite{Ch14} and in the thesis by N. Provost \cite{Pro} independently.

Let us take $\Gm=\SL_2(\ZZ)$, and assume that $v_1=v_2=1$, and
$w_1,w_2\in 2\ZZ_{\geq 0}$. This implies that the multiple integrals yield
polynomial functions   in the variable~$t $. 
We apply \eqref{rel2} with $z=x=\infty$
and $y=0$. With \eqref{trivrel}
\[ R_2(f_1,f_2;\infty,0;t) + R_2(f_1,f_2;0,\infty;t)
\= R_1(f_1;\infty,0;t) \, R_1(f_2;0,\infty;t)\,.\]
Using the binomial theorem, we see that $R_1(f;\infty,0;t)$ is a polynomial
in $t$ with coefficients that can be expressed in values of completed
$L$-functions. In a similar way, $R_2(f_1,f_2;\infty,0;t)$ is a polynomial
in $t$ with coefficients that can be expressed in values of a completed
multiple $L$-function of order $2$ as defined in \cite[(2.6)]{Ch14}. With
Lemma~\ref{lem-inv}
$R_2(f_1,f_2;0,\infty;\cdot) = R_2(f_1,f_2;\infty,0;\cdot)|_{-w}S$,
where $S=\matr0{-1}10$. In this way, the decomposition relation \eqref{rel2}
implies the equality of two polynomials. Comparing coefficients leads to
the relation in \cite[Theorem 3.1]{Ch14}.

This account is a simplification. The decomposition relations are valid for
the iterated integrals in~\eqref{Rdef0}, and lead for
$w_j \in 2\ZZ_{\geq 0}$ to polynomials in two variables. In \cite{Ch14}
Choie works in that generality.

\section{Formal series}\label{sect-fs}


 Manin \cite{Ma5, Ma6} has indicated a way to give structure to the
 decomposition relations of any order. His approach works in a general
 context of iterated integrals associated to cusp forms. The factors
 $(\tau_j-\nobreak t)^w$ of the kernel in \eqref{Rdef1} and
 $(\tau_j-\nobreak t_j)$ in \eqref{Rdef0} may be replaced by more general
 factors, for instance, by factors leading to iterated $L$-integrals as
 studied in~\cite{Ch14}. Here we use Manin's formalism for the iterated
 integrals in \eqref{Rdef1}.\medskip

We keep fixed $\ell$ combinations of a weight $w_j+2\in \RR$ and a
corresponding unitary multiplier system $v_j$. For a vector
$\ff=(f_1,\ldots,f_\ell) \in \prod_{j=1}^\ell \cf{w_j+2}(\Gm,v_j)$ of
length $\ell$ we form iterated integrals of arbitrary order
\be \label{mitint}
R_n(f_{m_1},f_{m_2},\ldots,f_{m_n}; y,x; t)\ee
for any choice $m=(m_1,\ldots,m_n) \in \{1,\ldots,\ell\}^n$ for any
$n\geq 0$. For $n=0$ we define this quantity to be~$1$. The same $f_j$ may
occur several times as $f_{m_i}$. So we do not get linearity in $f_j$. The
result is a holomorphic function on $\lhp$, and has at most polynomial
growth by Lemma~\ref{lem-pg}.

To formulate the $\Gm$-equivariance we put for $m=(m_1,m_2,\ldots,m_n)$
\be\label{vvww} \v(m) \;:=\; v_{m_1}v_{m_2}\cdots v_{m_n}, \qquad \w(m)
\;:=\; w_{m_1}+w_{m_2}+\cdots w_{m_n}\,. \ee
We consider the iterated integral in \eqref{mitint} as an element of
$V\bigl(\v(m),\w(m)\bigr)$. For the empty sequence $m=()$ we put
$V\bigl(\v(),\w() \bigr)=\CC$ with the trivial action $|_{1,0}$.
Multiplication follows the rule in~\eqref{mlta}. Lemma~\ref{lem-inv} can be
applied.

\rmrk{Power series in noncommuting variables} We choose $\ell$ spaces of
cusp forms $\cf{w_j+2}(\Gm,v_j)$ with $w_j+2>0$ and unitary multiplier
systems $v_j$, for $1\leq j \leq \ell$. We indicate this choice by the
symbol~$\A$. For this choice $\A$ we take $\ell$ noncommuting variables
$A_1,A_2,\ldots,A_\ell$.

By $\O(\A)$ we denote the set of formal power series in the $A_j$ for which
the coefficient of the monomial $A_{m_1}A_{m_2}\cdots A_{m_n}$
is in $V\bigl(\v(m),\w(m)\bigr)$ for each $m\in \{1,\ldots,\ell\}^n$. The
constant term is in $V\bigl(\v(),\w()\bigr)=\CC$. The relation \eqref{mlta}
implies that $\O(\A)$ is a ring.

\rmrk{Formal series associated to vectors of cusp forms} Following Manin we
combine all iterated integrals in~\eqref{mitint} as coefficients of an
element of the ring $\O(\A)$. Let
\be\label{SA} \cf \A(\Gm) \= \prod_{j=1}^\ell \cf{w_j+2}(\Gm,v_j)\,.\ee
For $\ff=(f_1,\ldots,f_\ell) \in\cf\A(\Gm)$   define the formal series
$J(\ff;y,x;t) \in \O(\A)$ by
\badl{Jdef} J(\ff;& y,x;t) \= 1\\
&\hbox{}+\sum_{n\geq 1} \sum_{m_1,\ldots,m_n \in \{1,\ldots,\ell\}}
 R_n(f_{m_1},f_{m_2},\ldots,f_{m_n};y,x; t)\, A_{m_1}A_{m_2}\cdots
A_{m_n}\,.\eadl

\rmrks
\itmi $J(\ff;z,w;\cdot)$ is an invertible element of $\O(\A)$ since it has a
non-zero constant term.

\itm
The coefficients  $ R_n(f_{m_1}, \ldots,f_{m_n};y,x;t)$ are continuous 
functions of  
$y,x\in \uhp^\ast$,   and are holomorphic in $x,y\in\uhp$.  

\itm The $A_j$   codes  for the space $\cf{w_j+2}(\Gm,v_j)$. This approach
differs from Manin's in \cite[\S2]{Ma5}. There the formal variables code
for linearly independent elements of the space
$\prod_j \cf{w_j+2}(\Gm,v_j)$.

\rmrk{Action of $\Gm$}We have define an action of $\Gm$ on $\O(\A)$ by the
action $|_{\v(m), - \w(m)}$ on the coefficient of $A_{m_1}\cdots A_{m_n}$.
Lemma \ref{lem-inv} implies the relation
\be\label{Jinv}
J(\ff;\gm^{-1}y,\gm^{-1}x;\cdot) = J(\ff;y,x;\cdot)|\gm\qquad\text{
for each }\gm\in \Gm\,.\ee

\rmrk{Multiplication properties}These formal series satisfy for
$z,y,x \in \uhp^\ast$:
\begin{align}
\label{mp-id}
J(\ff;x,x ; t )&\=1\,,
\displaybreak[0]\\
\label{mp-inv}
J(\ff;x,y ; t)&\= J(\ff;y,x ; t )^{-1}\,,\displaybreak[0]\\
\label{mp-mult}
J(\ff;z,x ; t )&\= J(\ff;z,y; t )\, J(\ff;y,x ; t)\,.
\end{align}
We will prove a more general result in Proposition~\ref{prop-Jprop}.

These relations encapsulate infinitely many relations between multiple
period integrals. The reader who takes the trouble to compare the
coefficients of $A_1 A_2$ in \eqref{mp-mult} obtains the
relation~\eqref{rel2}. Similarly, relation \eqref{rel3} is given by the
coefficient of $A_1 A_2 A_3$.

\subsection{Commutative example}\label{sect-ceJ}In the modular case we may
look at $w=\frac12 N-2$ for some $N\in \ZZ_{\geq 1}$. As the corresponding
multiplier system we choose $v_{N/2}$ determined by
\be v_{N/2}\matc1101\=e^{\pi i N/12}\,,\qquad v_{N/2} \matr0{-1}10\= e^{-\pi
i N/4}\,. \ee
For $1\leq N \leq 24$ the space of cusp forms is one-dimensional:
$\cf{N/2+2}\bigl( \Gm(1), v_{N/2}\bigr) \= \CC \; \eta^N$, where
$\eta(\tau) = q^{1/24}\prod_{n\geq 1} (1-\nobreak q^n)$,
$q=e^{2\pi i \tau}$, is the Dedekind eta-function.

We take $\ell=1$, with $w_1=\frac N2-2$ and multiplier system $v_{N/2}$. The
ring $\O(\A)$ is a \emph{commutative} ring of formal power series in one
variable $A$. The coefficient of $  A^ m  $ is in the $\Gm(1)$-module
$V\bigl(v_{mN/2}, \frac m2N -\nobreak 2m\bigr)$.

If we take $1\leq N \leq 24$, then with $\ff=(\eta^N)$ we get
in~\eqref{Jdef}
\be\label{JN} J(\ff;y,x;t)
\= 1+\sum_{n\geq 1} R_n\bigl((\eta^N)^{\times n} ;y,x;t)
A^n\,, \ee
where $(\eta^N)^{\times n}$ means a sequence of $n$ copies of $\eta^N$.

If $N>24$ we still  
can work with   $\ff=(f)$, but now $f$ need not be a multiple of
$\eta^N$.

\section{From cusp forms to noncommutative cohomology}\label{sect-ncc}

Manin uses relation \eqref{mp-mult} to associate a noncommutative cocycle to
the vector $\ff=(f_1,\ldots,f_\ell)$ of cusp forms. We first reformulate
Manin's description \cite[\S1]{Ma5} of noncommutative cohomology for a
right action, and then determine the map from vectors of cusp forms to
noncommutative cohomology.

\rmrk{Noncommutative cohomology} Let $G$ and $N$ be groups, written
multiplicatively, and suppose that for each $g\in G$ there is an
automorphism $n\mapsto n|g$ of $N$ such that the map $g\mapsto |g$ is an
antihomomorphism from $G$ to the automorphism group $\mathrm{Aut}(N)$, ie.,
$n|(gh) = \bigl(n|g)|h$ for $n\in N$ and $g,h\in G$.

A map $ \rho  :G \rightarrow  N $ is called a $1$-cocycle if it satisfies
\be \rho_{gh} \=(\rho_g|h)\, \rho_h
\quad\text{ for all }g,h\in G\,.\ee
The set of such cocycles is called $Z^1(G;N)$. It is not a group.
Nevertheless it contains the special element $1:g\mapsto 1$.

The group $N$ acts on $Z^1(G;N)$ from the left, by $ \rho \mapsto{}^n \rho$
defined by
\be\label{nd} {}^n \rho_g = (n|g)\,\rho_g\, n^{-1}\,. \ee
The cohomology set $H^1(G;N)$ is the set of $N$-orbits in $Z^1(G;N)$ for
this action. The orbit of the cocycle $g\mapsto 1$ is called the set of
coboundaries $B^1(G;N)$.

\rmrk{Noncommutative cocycles attached to a sequence of cusp forms} As the
group $N$ we use the subgroup $\N(\A)$ of the group of those units in
$\O(\A)^\ast$ that have constant term equal to~$1$. The series
$J(\ff; y,x ;\cdot)$ in \eqref{Jdef} is an element of~$\N(\A)$.

Following Manin we define for $\ff=(f_1,\ldots,f_\ell) \in \cf\A(\Gm)$ and
$x\in \uhp^\ast$
\be \Ps(\ff)^x_\gm (t)= J(\ff;\gm^{-1}x,x;t)
\,.\ee
The properties \eqref{Jinv} and \eqref{mp-mult} imply that this defines a
noncommutative cocycle $ \Ps(\ff)^x  \in Z^1\bigl( \Gm;\N(\A)\bigr)$, and
that its cohomology class $\Coh_\A(\ff)\in H^1\bigl(\Gm;\N(\A)\bigr)$ does
not depend on the choice of the base-point~$x$. We write
$\Ps(\ff)=\Ps(\ff)^\infty$.

\begin{prop}\label{prop-inj}The map
\be \Coh_\A: \cf\A(\Gm) \rightarrow H^1\bigl(\Gm;\N(\A)\bigr)\ee
is injective.
\end{prop}
\begin{proof}
Suppose that the cocycles $\Ps(f_1,\ldots,f_\ell)$ and
$\Ps(f_1',\ldots,f_\ell')$ are in the same cohomology class. Then there is
$n\in \N(\A)$ such that for all $\gm\in \Gm$
\be \label{ane} \Ps(f_1',\ldots,f_\ell')_\gm = (n|\gm)\,
\Ps(f_1,\ldots,f_\ell)_\gm\, n^{-1}\,.\ee
We denote the coefficient of $A_j$ in $n$ by $n_j\in V(v_j,w_j)$. In
relation \eqref{ane} we consider only the constant term and the term with
$A_j$, and work modulo all other terms:
\[ 1 - \ps_{f_j',\gm}\, A_j \;\equiv\; \bigl(1+n_j\, A_j\bigr)\; \bigl(
1-\ps_{f_j,\gm}\, A_j \bigr)\, \bigl( 1- n_j\, A_j\bigr)\,, \]
Taking the factor of $A_j$ gives the following:
\[
- \ps_{f_j',\gm} \= n_j|_{ v_j,-w_j }\gm - \ps_{f_j,\gm} - n_j\,. \]
In other words, $\ps_{f_j'}$ and $\ps_{f_j}$ differ by a coboundary. We have
used the noncommutative relation \eqref{ane} in $\N(\A)$ to get a
commutative relation in $V(v_j,w_j)$.

By the Theorem of Knopp and Mawi (Theorem~\ref{thm-KM}) we conclude that
$f_j'=f_j$ for all $j$. Hence $\Coh_\A$ is injective.
\end{proof}
\rmrks \itmi Implicit in the proof is the quotient of $\A$ by the ideal
generated by all monomials in the $A_j$ with degree $2$. The corresponding
quotient of $\N(\A)$ is isomorphic to the direct sum of  the $V(v_j, w_j)$. 
  \itm The injectivity of the map from cusp forms to cocycles is a point in
common   for  this result,   the theorem of Knopp and Mawi, and the
classical Eichler-Shimura result. The bijectivity in the theorem of Knopp
and Mawi is not shared by the classical result, where conjugates of
cocycles also determine cohomology classes. In the next section we will see
that the whole group $H^1\bigl( \Gm;\N(\A) \bigr)$ can be described with
cusp forms, but in a more complicated way than by the map~$\Coh_\A$.

\subsection{Commutative example}\label{sect-cePsi}In~\S\ref{sect-ceJ}
we considered the case that $\ell=1$. Then $\N(\A)$ is a commutative group,
and $H^1\bigl(\Gm;\N(\A)\bigr)$ is a cohomology \emph{group}.

When $\Gm=\Gm(1)$, with the choices and notations indicated in
\S\ref{sect-ceJ}, the cocycle $\Ps(\eta^N)$ vanishes on $\matc 1101$
(hence may be called a parabolic cocycle), and is determined by its value on
$S=\matr0{-1}10$:
\be\label{Ps-mod} \Ps(\eta^N)_S(t) \= J(\eta^N;0,\infty;t)
\= 1+\sum_{n\geq 1} R_n\bigl( (\eta^N)^{\times n};0,\infty;t\bigr)\, A^n\,.
\ee
The coefficient of $A^n$ is an iterated period integral of $\eta^N$. The
cocycle satisfies the well known relations
$\bigr(\Ps(\eta^N)_S|S\bigr) \Ps(\eta^N)_S \= 1$ and
$\Ps(\eta^N)_S = \Ps (\eta^N)_S|T' \; \Ps(\eta^N)_S |T $, with
$T'=\matc1011=TST$.

\section{Noncommutative cocycles and collections of cusp
forms}\label{sect-scf}
The proof of proposition~\ref{prop-inj} is based on the fact that the vector
of cusp forms $\ff$ can be recovered from the terms of degree $1$ in the
formal series $J(\ff;\gm^{-1}\infty,\infty;t)$. In this section we
associate to collections of cusp forms noncommutative cocycles of a more
general nature.

We keep fixed the choice $\A$ of positive weights $w_1+2,\ldots, w_\ell+2$
and corresponding multiplier systems $v_1,\ldots, v_\ell$. To each monomial
$B=A_{m_1}\cdots A_{m_d}$ in $\O(\A)$ we associate the   shifted  weight
$\w(B):=\w(m)$ and the multiplier system $\v(B):=\v(m)$ as defined
in~\eqref{vvww} for $m=(m_1,\ldots,m_d)\in \{1\ldots,\ell\}^d$. So $\w(B)$
and $\v(B)$ depend only on the factors $A_{m_i}$ occurring in $B$, not on
their order.

\begin{defn}\label{cfA} We define the \emph{degree} $d(B)$ of the monomial
$B=A_{m_1}\cdots A_{m_d}$ as the number $d$ of factors $A_j$
($1\leq j \leq \ell$) occurring in it.

Let $\B(\A)$ be the set of all monomials $B$ in $A_1,\ldots, A_\ell$ with
$d(B)\geq 1$ for which $\cf{\w(B)+2}\bigl( \Gm,\v(B)\bigr)\neq \{0\}$. We
put
\be\label{cfAdef} \cf{}(\A;\Gm) \;:=\; \ \prod_{B \in \B(\A)}
\cf{\w(B)+2}\bigl(\Gm,\v(B)\bigr)\,.\ee
\end{defn}

\rmrks \itmi The space of cusp forms $\cf{\w(B)+2}\bigl(\Gm,\v(B)\bigr)$ may
be zero. In fact, this is necessarily the case if $\w(B)\leq -2$. For
$\w(B)>-2$ it may also happen to be zero, depending on $\Gm$ and~$\v(B)$.

\itm The set $\B(\A)$ is often infinite. We recall that elements of infinite
direct sums of vector spaces have zero components at all but finitely many
$B\in \B(\A)$. Here we use the product. Its elements may have non-zero
components for all~$B$.

\itm We denote elements of $\cf{}(\A;\Gm)$ by $\h$, with component $\h(B)$
in the factor corresponding to the monomial~$B$.

\itm There may be more than one monomial $B$ for which
$\cf{\w(B)+2}\bigl(\Gm,\v(B)\bigr)$ is equal to a given space of cusp
forms. See \ref{iN4} below for an example where this happens for infinitely
many monomials.

\itm\label{ff} The space $\cf\A(\Gm)= \prod_{j=1}^\ell \cf{w_j+2}(\Gm,v_j)$
in~\eqref{SA} may be considered as a subspace of $\cf{}(\A;\Gm)$. To do
this we define for a given $\ff=(f_1,\ldots,f_\ell) \in \cf\A(\Gm)$ the
 element $\h \in \cf{}(\A;\Gm)$ by
\[ \h(A_j)\=f_j\,, \qquad \h(B) \=0\quad\text{ if } d(B)\geq 2\,.\]

\itm In the \emph{commutative case} $\ell=1$ we have
$\B(\A) \subset \bigl\{ A^n \;:\; n \in \ZZ_{\geq 1}\bigr\}$. We consider
three specializations of the example in~\S\ref{sect-ceJ}.
\begin{enumerate}
\item[(f1)]\label{iN24} Take $N=24$. So $
\B(\A)=\bigl\{ A^n \;:\; n \in \ZZ_{\geq 1}\bigr\}, \w(B)=10n $ and
 $\v(B)= v_{12}=v_0=1$. Hence
\be\label{N24} \cf{}\bigl(\A;\Gm(1)\bigr) \= \prod_{n\geq 1}
\cf{10n+2}\bigl(\Gm(1),1\bigr)\,. \ee

\item[(f2)]\label{iN4} Take $N=4$. So $\w(A^n)=0$ for all $n\geq 1$ and the
space $\cf2\bigl(\Gm(1),v_{2n}\bigr) $ is equal to $\CC\, \eta^4$ if
$n\equiv 1 \bmod 6$ and zero otherwise. This implies that
$\B(\A)= \bigl\{A^n\geq 1\;:\; n \equiv 1\bmod 6\bigr\}$. Since
$\v(B)=v_n=v_2$ for $n\equiv 1 \bmod 6$, we obtain
\be\label{N4} \cf{}\bigl(\A;\Gm(1)\bigr) \= \prod_{n\geq 1,\; n\equiv 1
\bmod 6} \cf2\bigl( \Gm(1), v_2 \bigr) \,.\ee
\item[(f3)] \label{iN1} Take $N=1$. So $\w(B)=w_1 = -\frac 32$, and
$nw_1<-2$ for $n\geq 2$. Hence $\B(\A) = \{ A\}$, $\v(B)=v_{1/2}$, and
\be\label{N1}
 \cf{}\bigl(\A;\Gm(1)\bigr) \= \cf{1/2}\bigl(\Gm(1),v_{1/2}\bigr)
 \= \CC\,\eta \,,\ee
with $\eta(\tau) = e^{\pi i \tau/12}\prod_{n\geq 1}(1-e^{2\pi i n \tau})$.
\end{enumerate}

\begin{lem}For each $\h \in \cf{}(\A;\Gm)$ the following series converges
and defines an element of $\N(\A)$:
\badl{BFdef} J(\h;&y,x;t) \;:=\; 1 
\\
&\hbox{} + \sum_{n\geq 1}\, \sum_{B_1,\ldots,B_n\in \B(\A)}
R_n\bigl(\h(B_1),\h(B_2),\ldots,\h(B_n);y,x;t\bigr)\, B_1B_1\cdots
B_n\,. \eadl
\end{lem}
\begin{proof}The degree of $B_1B_2\cdots B_n$ is at least $n$. For
convergence in $\O(\A)$ there should be for each $D\geq 0$ only finitely
many terms with degree at most $D$. This restricts $n$ to $n\leq D$, and
the $B_j$ to monomials of degree bounded by $D$, of which there are only
finitely many.

The terms with $n\geq 1$ cannot contribute to the constant term, hence we
obtain an element of $\N(\A)$.
\end{proof}

\rmrks\itmi If $h(B_i)=0$ for some $i$ in the iterated integral in
\eqref{BFdef}, then the integral vanishes.  
In \eqref{BFdef}  
 we could have   restricted the $B_i$ in 
the sum by   the 
condition $\h(B_i)\neq 0$. In particular,
$J(0;y;x;t)= 1  $.

\itm Definition \eqref{BFdef} extends definition \eqref{Jdef}. If
$\ff \in \cf\A(\Gm)$ is considered as an element $\h\in\cf{}(\A;\Gm)$, as
in remark \ref{ff} to Definition~\ref{cfA}, then
\be J(\ff;y,x;t) \= J(\h;y,x;t)\,. \ee

\begin{prop}\label{prop-Jprop}For all $\h\in \cf{}(\A;\Gm)$, $\gm\in \Gm$,
$z,y,x\in \uhp^\ast$:
\begin{align}
\label{BJ-gm}J(\h;\gm^{-1}y,\gm^{-1}x;\cdot)&\=
J(\h;y,x;\cdot)|\gm\,,\\
\label{BJ-1}J(\h;x,x;t)&\=1\,,\\
\label{BJ-fe}J(\h;z,x;t)&\= J(\h;z,y;t)\, J(\h;y,x;t)\,,\\
\label{BJ-i}J(\h;x,y;t)&\= J(\h;y,x;t)^{-1}\,.
\end{align}
\end{prop}
\begin{proof}
The relations \eqref{BJ-gm} and \eqref{BJ-1} follow directly from
Lemma~\ref{lem-inv} and \eqref{Rdef1}. We will prove relation \eqref{BJ-fe}
in a sequence of lemmas, and finally will derive relation \eqref{BJ-i} from
relation~\eqref{BJ-fe}.

\rmrk{Relation \eqref{BJ-fe}} This relation holds in a general context of
iterated integrals; automorphic properties are not needed. Our proof
follows Manin closely, \cite[Proposition 1.2]{Ma6}. We first show relation
\eqref{BJ-fe} for $x,y,z\in \uhp$.

\begin{lem}\label{lem-diff}For $\h \in \cf{}(\A;\Gm)$ put
\be \label{Omdef}
 \Om(\h;z;t) \;:=\; \sum_{B\in \B(\A)}
(z-t)^{\w(B)}\, \h(B;z)\, dz\cdot B \,.\ee
This formal series of $\O(\A)$-valued differential forms converges, and for
$z\in \uhp$
\be\label{diff} d_z \, J(\h;z,x;t) \= \Om(\h;z;t)\, J(\h;z,x;t)\,.\ee
\end{lem}
\begin{proof} The sum in \eqref{Omdef} is infinite in most cases. The
convergence follows from the fact that the number of monomials with a given
degree is finite. The differential of a non-constant term in \eqref{BFdef}
is given by
\begin{align*}
d_z & R_n\bigl(\h(B_1),\h(B_2),\ldots,\h(B_n);z,x;t\bigr)\, B_1B_2 \cdots
B_n
\\
&\= \h(B_1;z)\, (z-t)^{\w(B_1)}\, B_1 \, R_{n-1}(\h(B_2),\ldots,
 \h(B_n);z,x;t)\, B_2 \cdots
 B_n\,.
\end{align*}
With a renumbering in the summation this gives~\eqref{diff}.
\end{proof}

\begin{lem}\label{lem-di}
$d_z\, J(\h;z,x;t)^{-1} \= - J(\h;z,x;t)^{-1}\,\Om(\h;z;t)$.
\end{lem}
\begin{proof} The inverse is defined by the relation
\[ 1 \= J(\h;z,x;t)^{-1}\, J(\h;z,x;t)\,. \]
Taking the differential of both sides gives with~\eqref{diff}
\[ 0\= \bigl(d_z J(\h;z,x;t)^{-1}\bigr) \, J(\h;z,x;t) + J(\h;z,x;t)^{-1}\,
\Om(\h;z;t)\, J(\h;z,x;t)\,. \]
Right multiplication by $J(\h;z,x;t)$ gives the relation in the lemma.
\end{proof}

\begin{lem}For fixed $x$ and $y$ put
$K(z) = J(\h;z,y;t)^{-1}\, J(\h;z,x;t)$. Then $K(z) = J(\h;y,x;t)$.
\end{lem}
\begin{proof}
With Lemmas \ref{lem-diff} and~\ref{lem-di} we find
\begin{align*}
 d_z K(z) &\= - J(\h;z,y;t)^{-1}\, \Om(\h;z;t)\, J(\h;z,x;t)\\
 &\qquad\hbox{}
  + J(\h;z,y;t)^{-1}\, \Om(\h;z;t)\, J(\h;z,x;t)\=0\,.
\end{align*}
Hence $K(z)$ is constant. Its value is
\[ K(y) \=J(\h;y,y;t)^{-1}\,J(\h(y,x;t) \= J(\h;y,x;t)\,, \]
by~\eqref{BJ-1}.
\end{proof}

\rmrk{Completion of the proof of relation \eqref{BJ-fe}} The relation
$K(z)= J(\h;y,x;t)=J(\h;z,y;t)^{-1}\, J(\h;z,x;t)$ implies the desired
result for $z,y,x\in \uhp$. By continuity it holds for
$z,y,x\in \uhp^\ast$.

\rmrk{Relation \ref{BJ-i}} This relation follows from \eqref{BJ-1}
and~\eqref{BJ-fe}. This ends the proof of Proposition~\ref{prop-Jprop}.
\end{proof}

\rmrk{Noncommutative cocycle}From \eqref{BJ-gm} and \eqref{BJ-fe} it follows
that for any $\h \in \cf{}(\A;\Gm)$
\be \label{PsiBdef}
\Psi(\h)_\gm \;:=\; J(\h;\gm^{-1}\infty,\infty;t)
\ee
defines a cocycle $\gamma \rightarrow \Psi(\h)_{\gamma}   $ in
$ Z^1\bigl( \Gm;\N(\A)\bigr)$.

If we replace $\infty$ in \eqref{PsiBdef} by another base-point
$x\in \uhp^\ast$ we get a cocycle in the same cohomology class. So
$\h \mapsto \Psi(\h)$ induces a map from $\cf{}(\A;\Gm)$ to the
noncommutative cohomology set $H^1\bigl(\Gm;\N(\A)\bigr)$, extending the
map $\Coh_\A$ in Proposition~\ref{prop-inj}.

The main result of this paper is the bijectivity of this map:
\begin{thm}\label{thm-surj}Let $\A$ denote the choice of finitely many
positive weights $w_1+\nobreak 2, \allowbreak w_2+\nobreak 2,
\allowbreak \ldots,w_\ell+2$ and corresponding multiplier systems
$v_1,\ldots,v_\ell$ of~$\Gm$.

For each noncommutative cohomology class  $c\in H^1 \bigl(\Gm; N(\A)\bigr)$  
there is a unique element $\h \in \cf{}(\A;\Gm)$ such that
$\Psi(\h) \in c$.
\end{thm}

\begin{proof}
The induction runs over $k\geq 0$. We start with a cocycle
$\X^0 \in Z^1\bigl(\Gm;\N(\A)\bigr)$, and replace it in the course of an
induction procedure by cocycles $\X^1,\X^2,\ldots$ in the same cohomology
class. During the induction we form a sequence
$\h_0, \h_1,\ldots,  \h_k,..$ of elements of $\cf{}(\A;\Gm)$, and a
strictly increasing sequence of integers $c_0,c_1,\ldots$. The connection
between the induction quantities $\X^k$, $\h_k$ and $c_k$ is given by the
requirement that at each stage of the induction the following conditions
hold:
\begin{enumerate}
\item[(H)]\label{h-cond} $\h_k(B) =0$ for all $B$ with $d(B)>c_k$.
\item[(XPs)]\label{diff-cond} If
$ \X^k_\gm -\Ps(\h_k)_\gm \= \sum_B a(\gm,B)\, B, $ where the sum $B$ runs
over all non-com-mu\-ta\-tive polynomials in $A_1,\ldots,A_\ell$, then for
 each $\gm\in \Gm$
\[ a(\gm,B)\=0 \quad\text{for all $B$ with }d(B)\leq k\,. \]
\end{enumerate}

Either at a certain stage $k$ in the induction procedure the process stops,
and we take $\h=\h_k$, or the process goes on indefinitely, in which case
we construct $\h$ as a limit of the~$\h_k$. In both cases we show that
$\Psi(\h) $ is in the cohomology class of the $\X^k$, and that the element
$\h$ is uniquely determined.

\rmrk{Start of the induction}For a given cocycle $\X^0  $ in the cohomology
class $c$ we put $\h_0=0$ and $c_k=0$. Then for all $\gm\in \Gm$
\[ \X^0_\gm-\Ps(\h_0)_\gm \= \X^0_\gm-1\]
has no constant term, and Conditions (H) and (XPs) are trivially satisfied.

\rmrkn{Has the end of the induction process been reached?}If
$\X^k = \Psi(\h_k)$ we have found a description of the class $c$ as
required in the theorem. This may happen already at the start of the
induction if $\X^0$ is the trivial cocycle $\gm\mapsto 1$.

\rmrk{Induction, choice of $c_{k+1}$}If the process has not ended, then the
difference $\Y^k:=\X^k-\Psi(\h_k)$ determines a non-zero map
$\gm \mapsto \Y^k_\gm$ from $\Gm$ to $\O(\A)$. It is not a cocycle.

We define $c_{k+1}$ as the minimum degree such that $\Y^k_\gm \in \O(\A)$
has non-zero terms of degree $c_{k+1}$ in $A_1,\ldots,A_\ell$ for some
$\gm\in \Gm$. Since Condition (XPs)
holds for $k$ we have $c_{k+1}>c_k$.

\rmrk{Cocycle relation} The cocycle relations for the noncommutative
cocycles $\X^k$ and $\Ps(\h_k)$ give
\begin{align}\nonumber
\Y^k_{\gm\dt} &\= \bigl( (\Y^k_\gm+\Ps(\h_k)_\gm)|\dt \bigr)\; \bigl(
\Y^k_\dt + \Ps(\h_k)_\dt\bigr)
- (\Ps(\h_k)_\gm|\dt) \, \Ps(\h_k)_\dt
\displaybreak[0]\\
\label{Yr}
&\= (\Y^k_\gm|\dt) \, \Ps(\h_k)_\dt + (\Ps(\h_k)_\gm|\dt)\, \Y^k_\dt +
(\Y^k_\gm|\dt)\, \Y^k_\dt\,.
\end{align}
 By Condition (XPs) and the choice of $c_{k+1}$ the element
 $\Y^k_\gm\in \O(\A)$ has no terms with degree less than $c_{k+1}$. We
 denote by $\bar\Y^k_\gm$ the sum of the terms of $\Y^k_\gm$ with exact
 degree $c_{k+1}$. We consider relation \eqref{Yr} modulo terms with degree
 strictly larger than~$c_{k+1}$:
\be \bar \Y^k_{\gm\dt} \;\equiv\;
(\bar \Y^k_\gm|\dt)\, \Ps(\h_k)_\dt + (\Ps(\h_k)_\gm|\dt)\, \bar \Y^k_\dt +
0\,.\ee
In the two products only the constant term $1$ of $\Ps(\h_k)_\gm$ and
$\Ps(\h_k)_\dt$ is relevant, and we obtain
\be\label{bycr} \bar \Y^k_{\gm \dt } \= \bar \Y^k_\gm|\dt + \bar \Y^k_\dt\,.
\ee
So the non-commutative cocycle relations for $\X^k$ and $\Ps(\h_k)$ imply
that $\gm \mapsto \bar\Y^k_\gm$ is a commutative cocycle with values in the
additive group of~$\O(\A)$.

The elements $\bar \Y^k_\gm$ have the form
\be\label{bYk} \bar \Y^k_\gm \= \sum_{n=1}^K \ph^n_\gm \, C_n\,,\qquad
C_n\=A_{p_{n,1}}A_{p_{n,2}}\cdots
A_{p_{n, c_{k+1}}} \ee
with $\ph^n_\gm \in V\bigl( \v(C_n),\w(C_n)\bigr)
$. The $C_n$ have degree $c_{k+1}$ in $A_1,\ldots,A_\ell$. For each $n$
there is some $\gm\in \Gm$ for which $\ph^n_\gm\neq 0$. Relation
\eqref{bycr} implies that each component of $\bar \Y^k$ is a cocycle:
$\ph^n\in Z^1\bigl(\Gm;V(\v(C_n),\w(C_n))\bigr)$. By Theorem~\ref{thm-KM}
there exist $a_n\in V(\v(C_n),\w(C_n))$ and unique cusp forms
 $ g_n \in \cf{\w(C_n)+2}\bigl(\Gm,\v(C_n)\bigr)$ such that for all
 $\gm\in \Gm$
\be\label{phipa} \ph^n_\gm = - \ps_{g_n,\gm} + a_n
|_{\v(C_n),-\w(C_n)}(\gm-1)\,.\ee

\rmrk{Induction, choice of $\X^{k+1}$}Take
\be\label{Hk} H_k \= 1 -\sum_{n=1}^K a_n\, C_n\,. \ee
This is an element of $\N(\A)$. We define the cocycle $\X^{k+1}$ in the same
class as $\X^k$ in the following way:
\be\label{XK1d} \X^{k+1} _\gm \= (H_k|\gm)\; \X^k_\gm \; H_k^{-1}\,.\ee

\rmrk{Induction, choice of $\h_{k+1}$}It may happen that
$C_n\not \in \B(\A)$ for some $n\in \{1,\ldots,  k  \}$. Then
$\cf{\w(C_n)+2}\bigl( \Gm, \v(C_n)\bigr)=0$, and $\ph^n$ is a coboundary
and $ g_n =0$.

By Condition (H) we have $\h_k(C_n)=0$ for $1\leq n \leq K$. We construct
$\h_{k+1}$ from $\h_k$ by taking $\h_{k+1}(C_n)= g_n $ for those $n$ for
which $C_n\in \B(\A)$, and $\h_{k+1}(B)=\h_k(B)$ otherwise. So
$\h_k(B)=\h_{k+1}(B)$ for all $B$ with $d(B)>c_{k+1}$, and Condition (H)
stays valid for $k+1$. If $C_n \not\in \B(\A)$ for all $n$, then
$\h_{k+1}=\h_k$.

\rmrk{Induction, check of Condition {\rm (XPs)} for $k+1$} Modulo terms of
order larger than $c_{k+1}$:
\begin{align}\nonumber
\X^{k+1}_\gm &\;\equiv\; (1-\sum_n a_n|\gm\, C_n) \; (\Ps(\h_k)_\gm + \bar
\Y^k_\gm)\;
(1 + \sum_n a_n\, C_n)\\
\nonumber
&\qquad\qquad \qquad\qquad \qquad \qquad\qquad \qquad
\bigl(\text{\eqref{XK1d}, (XPs)}\bigr)
\displaybreak[0]
\\\nonumber
&\;\equiv\; \Ps(\h_k)_\gm + \bar \Y^k _\gm
- \sum_n a_n|\gm \, C_n + \sum_n a_n\, C_n
\displaybreak[0]
\\\nonumber
&\;\equiv\; \Ps(\h_k)_\gm+ \sum_n \Bigl( - \ps_{ g_n,\gm} + a_n|(\gm-1) -
a_n|\gm + a_n \Bigr)\, C_n\\
\nonumber
&\qquad\qquad \qquad\qquad \qquad \qquad\qquad \qquad
\bigl(\text{\eqref{bYk}, \eqref{phipa}}\bigr)
\\
\label{k+1diff}
&\= \Psi(\h_k )_\gm + \sum_n R_1( g_n ;\gm^{-1}\infty,\infty)\, C_n\,.
\qquad\bigl(\text{\eqref{o1coc}}\bigr)
\end{align}

By \eqref{PsiBdef} and \eqref{BFdef} we have
\begin{align*}
\Ps(&\h_{k+1})_\gm \=1 + \sum_{m\geq 1} \sum_{B_1,\ldots,B_m\in \B(\A)}\\
&\qquad\hbox{} R_m(\h_{k+1}(B_1),
\h_{k+1}(B_2),\ldots,\h_{k+1}(B_m);\gm^{-1}\infty,\infty;t)B_1 B_2 \cdots
B_m \,,
\end{align*}
in which we can leave out the terms in which a $B_i$ occurs with
$d(B_i)>c_{k+1}$, by Condition (H). If we leave out the terms with a $B_i$
for which $d(B_i)>c_k$, we obtain $\Ps(\h_k)_\gm$. If there is a $B_i$ with
$d(B_i)>c_{k}$ this is one of the $C_n$ in \eqref{bYk}, with
$d(B_i)=d(C_n)=c_{k+1}$. Working modulo terms with degree larger than
$c_{k+1}$ we obtain
\begin{align}\nonumber
\Ps(\h_{k+1})_\gm& - \Ps(\h_k)_\gm \;\equiv\; \sum_{ 1\leq n \leq K,\;
C_n\in \B(\A)} R_1(\h_{k+1}(C_n);\gm^{-1}\infty,\infty;t)\, C_n
\\
\label{kk1}&\= \sum_{1\leq n \leq K,\; C_n\in \B(\A)} R_1( g_n
;\gm^{-1}\infty,\infty;t)\, C_n \,.
\end{align}
A comparison of \eqref{k+1diff} and \eqref{kk1} gives Condition
(XPs) for $k+1$.

\rmrk{The induction may halt} It may happen that the induction stops at
stage $k$; namely, if $\X^k=\Psi(\h_k)$. Then we have found an element
$\h=\h_k\in \cf{}(\A;\Gm)$ such that $\Psi(\h)$ is in the class~$c$.

\rmrk{The induction may have infinitely many steps} It may also happen that
we have obtained after infinitely many steps an infinite sequence of
cocycles $\X^k$ in the class $c$, an infinite sequence of $\h_k$, and a
strictly increasing sequence of~$c_k$ satisfying Conditions (H) and~(XPs)
for all~$k$. For each monomial $B\in \B(\A)$ there is at most one $k$ such
that $h_n(B)=0$ for $n\leq k$, and $h_n(B) =h_{k+1}(B)$ for $n\geq k+1$. So
the componentwise limit $\h:=\lim_{k\rightarrow \infty} \h_k$ exists.

The construction of the sequence $(\X^k)_k$ implies that
\[ \X^k_\gm = \bigl((H_{k-1}H_{k-2}\cdots
H_0)|\gm\bigr)\; \X^0_\gm \;
(H_{k-1}H_{k-2}\cdots H_0)^{-1}\,,\]
with $H_k$ as in \eqref{Hk}. The infinite product $H=\cdots H_2\,H_1\,H_0$
converges in $\N(\A)$, since each $H_k$ equals $1$ plus a term in degree
$c_{k+1}$. Similarly, $J(\h;\gm^{-1}\infty,\infty;t) $ is the limit of the
$J(\h_k;\gm^{-1}\infty,\infty;t)$ as $k\rightarrow \infty$, since enlarging
$k$ we change only terms of degrees larger than $c_k$. Condition (XPs) is
valid for all $k$, so the conclusion is that in the limit
\be (H|\gm)\, \X^0 \, H^{-1} \= \Psi(\h)\,.\ee

\rmrk{Uniqueness}Let $\X^0 = \Ps(\h')$ for some $\h'\in \cf{}(\A;\Gm)$. We
claim that in the induction procedure described above applied to this
cocycle $\X^0$ we have at each stage
\begin{align}
\label{c1}
\h_k(B)&\= \h'(B)\quad\text{ for all $B\in \B(\A)$ with } d(B) \leq c_k\,;
\\
\label{c2}
\X^k&\;\equiv \Ps(\h') \quad\text{ modulo terms of degree larger than
$c_k$}\,.
\end{align}
This is true at the start of the induction (use $c_0=0$).

At stage $k$ the non-zero terms with lowest degree in
\[ \Y^k_\gm = J(\h';\gm^{-1}\infty,\infty;t) - J(\h_k
;\gm^{-1}\infty,\infty;t)\]
are due to the $C\in \B(\A)$ with degree equal to $c_{k+1}$. So
\be \bar \Y_\gm \= \sum_{C\in \B(\A),\; d(C)=c_{k+1}} R_1(\h'(B)
;\gm^{-1}\infty,\infty;t)\, C\,. \ee
Let us number the monomials in this sum as $C_1,\ldots, C_K$. Then
$\ph^n_\gm$ in \eqref{phipa} is equal to $-\psi_{\h'(C_n),\gm}$, and
$a_n=0$, $H_k=1$. This implies that $\h_{k+1}(C) = \h'(C)$ for the
monomials $C\in \B(\A)$ with degree $c_{k+1}$, and
$\X^{k+1}=\X^k\equiv \Ps(\h')$ modulo terms with degree larger than
$c_{k+1}$.

 At the end of the induction process we have $\B=\B'$, thus obtaining the
 uniqueness.
\end{proof}

\subsection{Concluding remarks}\itmi Manin \cite{Ma5, Ma6} used formal
series similar to those in \eqref{Jdef} to get a simple description of
relations among iterated integrals. In that approach the noncommutative
cohomology set $H^1\bigl( \Gm;\N(\A)\bigr)$ is a tool. In this paper we
further study the cohomology set $H^1\bigl(\Gm;\N(\A)\bigr)$.

\itm One may apply the approach of this paper to weights in $\ZZ_{\geq 2}$
and trivial multiplier systems. Then the iterated integrals are polynomial
functions. These are in a much smaller $\Gm$-module than the functions with
polynomial growth that we employ. The consequence is that the theorem
analogous to the theorem of Knopp and Mawi (Theorem~\ref{thm-KM}) does not
hold. Cocycles attached to conjugates of holomorphic cusp forms have to be
considered as well (see \cite{Kn74}). However, iterated integrals in which
occur both holomorphic and antiholomorphic cusp forms satisfy more
complicated decomposition relations. We think that Manin's formalism does
not work in that situation.

\itm The same problem occurs if we use the modules in Theorems B and D of
\cite{BCD}, unless we pick the weights $w_j+2$ in such a way that the
elements $\w(C)$ that occur in the sums defining $J(\B;y,x;t)$ are never in
$\ZZ_{\geq0}$.

\itm We work with iterated integrals of the type in \eqref{Rdef1}. Equation
\eqref{Rdef0} defines iterated integrals depending on variables
$t_1,\ldots,t_\ell$ all running independently through the lower half-plane.
It would be nice to have results for the corresponding non-commutative
cocycles. They can be defined, Proposition~\ref{prop-inj} goes through. We
did not manage to adapt the proof of Theorem~\ref{thm-surj} to cocycles of
this type. The problem is to construct formal sequences of the type in
\eqref{BFdef} such that they have the lowest degree terms in a prescribed
summand in the decomposition of the tensor products
$V(v_1,w_1)\otimes \cdots \otimes V(v_\ell,w_\ell)$ into submodules.

\iflitnum
\newcommand\bibit[4]{
\bibitem {#1}#2: {\em #3;\/ } #4}
\else
\newcommand\bibit[4]{
\bibitem[#1] {#1}#2: {\em #3;\/ } #4}
\fi

\end{document}